\documentclass[12pt, reqno]{amsart}
\usepackage{amsfonts,mathrsfs, amssymb}
\theoremstyle{plain}
\begingroup
\newtheorem{theorem}{Theorem}[section]
\newtheorem{lemma}[theorem]{Lemma}
\newtheorem{proposition}[theorem]{Proposition}
\newtheorem{openpr}{Open Problem}
\newtheorem{corollary}[theorem]{Corollary}

\newtheorem{definition}[theorem]{Definition}
\newtheorem{rmk}[theorem]{Remark}

\endgroup

\theoremstyle{remark}
\begingroup
\endgroup

\mathchardef\emptyset="001F

\numberwithin{equation}{section}

\newcommand{\op}[1]{{\rm{#1}}}

\newcommand{\de}{\partial}

\newcommand{\e}{\varepsilon}

\newcommand{\R}[1]{{\mathbb R^{#1}}}

\title[Singular $L^p$-bundles]
{Weak closure of Singular Abelian $L^p$-bundles in $3$ dimensions}

\author[Mircea Petrache,Tristan Rivi\`ere ]{Mircea Petrache, Tristan Rivi\`ere}

\begin{document}

\begin{abstract}
We prove the closure for the sequential weak $L^p$-topology of the class of vectorfields on $B^3$ having integer flux through almost every sphere. We show how this problem is connected to the study of the minimization problem for the Yang-Mills functional in dimension higher than critical, in the abelian case. 
\end{abstract}
 \maketitle

\section{Introduction}
In this work we consider the class $L^p_{\mathbb Z}(B^3, \mathbb R^3)$ of vectorfields $X\in L^p(B^3, \mathbb R^3)$ such that
$$
\int_{\de B_r^3(a)} X\cdot \nu \in\mathbb Z,\quad\forall a\in B^3, \text{ a.e. }r<\op{dist}(a,\de B^3), 
$$
where $\nu:\de B_r^3(a)\to S^2$ is the outward unit normal vector.\\

We observe that for $p\geq 3/2$ this class reduces to the divergence-free vectorfields, and therefore we reduce o the ``interesting'' case $p\in[1,3/2[$. It is clear that this class of vectorfields is closed by strong $L^p$-convergence (see Lemma \ref{clostrong}). We are interested in the closedness properties of $L^p_{\mathbb Z}(B^3, \mathbb R^3)$ for the sequential weak-$L^p$ topology, and our main result in the present work is the following:
\begin{theorem}\label{mainthm1}
 For $1<p<3/2$ the class $L^p_{\mathbb Z}(B^3, \mathbb R^3)$ is weakly sequentially closed. More precisely, whenever 
$$
X_k\in L^p_{\mathbb Z}(B^3, \mathbb R^3),\quad X_k\stackrel{\text{weak-}L^p}{\rightharpoonup}X_\infty,
$$
then $X_\infty\in L^p_{\mathbb Z}(B^3,\mathbb R^3)$.\\

For $p=1$ given any vector-valued Radon measure $X\in \mathcal M^3(B^3)$ where
$$
\mathcal M^3(B^3):=\{(\mu_1,\mu_2,\mu_3)|\;\mu_i\text{ signed Radon measure on }B^3\},
$$
we can find a sequence $X_k\in L^1_{\mathbb Z}(B^3, \mathbb R^3)$ such that $X_k\rightharpoonup X$ weakly in the sense of measures.
\end{theorem}

\section{Motivation: Yang-Mills theory in supercritical dimension}
In this section we show how theorem \ref{mainthm1} can be used in the framework of Yang-Mills theory.
Consider a principal $G$-bundle $\pi: P\to M$ over a compact Riemannian manifold $M$, and call $\mathcal A(P)$ the space of smooth connections on $P$. The celebrated \textit{Yang-Mills functional} $YM:\mathcal A\to\mathbb R^+\cup\{\infty\}$ is then defined as the $L^2$-energy of the curvature $F_A$ of $A$: 
$$
YM(A):=\int_M|F_A|^2d\;Vol_g.
$$
Critical points of this functional are connections satisfying in a weak sense the Yang-Mills equations, which have been extensively studied in the conformal dimension $4$, due to the geometric invariants that they help define (see \cite{DoKr, FreedUh}). We will focus here on the study of the functional $YM$ from a variational viewpoint. We will just consider the two simplest and most celebrated cases $G=SU(2)$ (nonabelian case) and $G=U(1)$ (abelian case).

\subsubsection{The smooth class} 
We first observe that since the Yang-Mills equations could have singularities, it is not clear if the infimum
$$
\inf_{C^\infty_\varphi\ni A}YM(A)
$$
(the subscript ``$\varphi$'' means that we are fixing a smooth boundary datum $\varphi$) is attained. The ``natural'' way to extend the class where $YM$ is defined, would then be to allow more general $L^2$-curvatures. Since $F_A=dA+A\wedge A$, we would have to consider therefore $W^{1,2}$-regular connections, and $W^{2,2}$-change of gauge functions (for a detailed description of the theory of Sobolev principal bundles see for example \cite{Wehrheim, Kessel, Isobdl}).
\begin{rmk}
 The fact that by the usual Sobolev embedding theorem $W^{2,2}\hookrightarrow C^0$ only when the dimension of the domain is $<4$ implies that the topology of the bundles which we consider is fixed just in low dimension. We therefore call $n=4$ the \emph{critical} dimension in the study of the functional $YM$.
\end{rmk}
\subsection{A parallel between the study of harmonic maps $u:B^3\to S^2$ and that of $YM$ in dimension $5$}
The most celebrated problem in which the study of singularities in a variational setting was introduced, is the minimization of the Dirichlet energy
$$
E(u):=\int_{B^3}|\nabla u|^2dx\quad \text{ for }u:B^3\to S^2\text{ with }u|_{\de B^3}=\varphi.
$$
As above, the infimum
$$
\inf_{C^\infty_\varphi(B^3, S^2)}E
$$
is in general \emph{not} achieved, and in this case the natural space to look at would be the space of functions having one weak derivative in $L^2$, namely $W^{1,2}(B^3, S^2)$. We observe that for the functional $E$ the critical dimension would then be $2$, since $W^{1,2}(B^n, S^2)\hookrightarrow C^0$ just for $n<2$. The optimal result achieved in this case is Theorem \ref{schoenuh} below. Such a regularity result would be the main goal in the study of the functional $YM$ in dimension higher than critical (but we will see that extra difficulties arise when we deal with singular curvatures). 
\begin{theorem}[\cite{SU2}]\label{schoenuh}
 Given a smooth map $\varphi\in C^\infty(S^2, S^2)$, for any minimizer $u$ of $E$ in $W^{1,2}_\varphi(B^3, S^2)$ there exist finitely many points $a_1,\ldots, a_N\in B^3$ and numbers $d_1,\ldots,d_N =\pm 1$ such that 
$$
\begin{array}{c}
 u\in C^\infty(B^3\setminus\{a_1,\ldots,a_N\}, S^2),\\
\op{deg}(u, a_i)=d_i\text{ for all }i.
\end{array}
$$
\end{theorem}
The singularities around which $u$ realizes a nonzero integer degree as above 
\footnote{The realization of a degree around a point is a local topological obstruction for the strong approximability in $W^{1,2}$-norm \cite{Bethuel2, BCDH}. Global obstructions also play a role in approximability properties \cite{HangLin2}.}
are called \emph{topological singularities}. \\

If we want to consider an analogous topological obstructions for connections on bundles, we must start from the celebrated \emph{Chern-Weil theory} \cite{Zhang}, which describes topological invariants of bundles via characteristic classes represented in terms of curvatures of connections. The most prominent ``topological singularity'' notion arising in relation to Yang-Mills $SU(2)$-gauge theory in dimension $4$ is encoded into the second Chern class of the associated bundle. This homology class can be represented using the curvature of a smooth connection $A$ via the Chern-Weil formula
$$
c_2(P)=\left[-\frac{1}{8\pi^2}\op{tr}(F_A\wedge F_A)\right].
$$
In \cite{Uh3} it was proved that (in dimension $4$) under a boundedness condition on the $L^2$ norm of the curvature, we have $c_2(P)\in\mathbb Z$. Such integrality condition has a role which is analogous to the one played by the integrality of the degree of maps $g:S^2\to S^2$ in the study of harmonic maps in $W^{1,2}(B^3,S^2)$, and as such is useful to study the $YM$ functional in dimension $5$. More precisely, the strategy \cite{RivKess, Kessel} consists in introducing the analogous of the space of maps with topological singularities described in Theorem \ref{schoenuh}. This time one has to consider smooth bundles defined on the base manifold with some finite set of points removed
$$
\mathcal P:=\left\{
\begin{array}{c}
\text{principal }SU(2)\text{-bundles of the form}\\
P\to M\setminus \Sigma,\text{ for some finite set }\Sigma\subset M
\end{array}
\right\},
$$
and then take the curvaturs of smooth connections on these bundles, realizing integral Chern numbers on small spheres surrounding the singularities:
$$
\mathcal R^\infty:=
\left\{F_A\left|\;
\begin{array}{l}
A\text{ is a smooth connection on some }P\in\mathcal P,\\
c_2(P|_{\de B(x,\e)})\in\mathbb Z\setminus\{0\}\text{ for }x\in\Sigma, \forall\e<\op{dist}(x, \de B\cup\Sigma)
\end{array}
\right.\right\}.
$$
Motivated by the above analogy, we can say that \emph{the above class $\mathcal R^\infty$ should be contained in any candidate for a class of critical points of $YM$ in dimension }$5$, as already noted in \cite{Kessel, RivKess}. If instead of considering the whole $\mathcal R^\infty$, we fix the number, degree and position of the singularities, a gap phenomenon arises, in analogy to \cite{BCL}, as described in \cite{Iso2, Iso3}.

\subsection{The basic difficulty: singularities of bundles}
What is it that forbids to continue the analogy with harmonic maps, and to prove a regularity result like Theorem \ref{schoenuh} for $YM$ in dimension $5$? To answer this question, we recall the two main ingredients without which such a result is not possible:
\begin{enumerate}
 \item[(A)]\label{A} A good \emph{variational setting}, i.e. the presence of a class which contains the maps with topological singularities and in which a minimizing sequence of $E$ has a converging subsequence. It is shown in \cite{Bethuel1, Bethuel2} that for harmonic maps, this class is $W^{1,2}$ with the sequaential weak topology, indeed \emph{
$$
C^\infty_\phi(B^3, S^2)\text { is weakly sequentially dense in }W^{1,2}(B^3, S^2),
$$
}and the wanted compactness property is a consequence of the Banach-Alaoglu theorem. In $W^{1,2}(B^3, S^2)$ therefore, the existence of a minimizer for $E$ with fixed boundary datum is clear, and therefore the \emph{existence of weak solutions} for the equation of critical points of $E$ is established.
\item[(B)]\label{B} An \emph{$\e$-regularity theorem}, i.e. the implication \textit{
$$
\begin{array}{c}
E(u)\leq\e\text{ on }B_1\text{ and }u\text{ is a minimizer of }E\\
\Rightarrow\\
u\text{ is H\"older on }B_{1/2}.
\end{array}
$$}
This kind of result is used to prove the discreteness of the singularity set, which is the main difficulty in the proof of Theorem \ref{schoenuh}.
\end{enumerate}
We see that already the ingredient (A) (which is needed in order to formulate (B)) is problematic in the case of the $YM$ functional on bundles, since nothing shows that \emph{a priori} the singularities of a minimizer should not accumlate, and it is not clear how to define a bundle with \emph{accumulating topological singularities}, at least if we stick to the differential-geometric definition of a bundle. We can formulate the following general problem:
\begin{openpr}[Right variational setting for the Yang-Mills theory in dimension $5$]\label{op}
Find a topological space 
\begin{enumerate}
\item which includes the class $\mathcal R^\infty$ of curvatures with finitely many singularities
\item in which minimizing sequences for $YM$ are compact.
\end{enumerate}
\end{openpr}
The natural candidate for a space solving the above problem in the nonabelian case is the space of \emph{$L^2$-curvatures on singular $SU(2)$-bundles} defined in \cite{RivKess}, Definition III.2. Such class clearly contains $\mathcal R^\infty$; Open Problem \ref{op} is then equivalent to asking to find a suitable topology on this natural class such that minimizing sequences for $YM$ are compact. Not only is the above problem still open, but \emph{it is not known whether for all smooth boundary data the infimum of $YM$ is achieved, even in the class of $L^2$-curvatures of \cite{RivKess}.}

\subsection{Main Result}
In this work we obtain the good setting described above (thereby solving Open Problem \ref{op}), in the simpler case of \emph{abelian bundles}, i.e. bundles with gauge group $U(1)$.\\

Smooth principal $U(1)$-bundles on a $2$-manifold $\Sigma$ correspond to hermitian complex line bundles, and are classified by the integer given (see \cite{MilnorStasheff} for example) by the first Chern class, again expressable via the Chern-Weil theory by
$$
c_1(P):=\frac{1}{2\pi}\int_\Sigma F_A.
$$
We are lead by the analogy with the above discussion about the nonabelian case in dimension $5$, to consider the $YM$ functional in dimension $3$, where the point singularities would be classified by the value of 
$$
[c_1(P|_{\de B_\e(x)})]\in H^2(B^3,\mathbb Z)
$$
for any sufficiently small sphere $\de B_\e(x)$ near an isolated singular point $x$. The corresponding classes $\mathcal P$ and $\mathcal R^\infty$ for $U(1)$-bundles, are defined as above, substituting ``$SU(2), c_2(P|_{\de B^5(x)})$'' with ``$U(1),c_1(P|_{\de B^3(x)})$''.
\begin{rmk}
  We are helped in our approach by the fact that in the abelian case the curvature of a $U(1)$-bundle over $B^3$ projects to a well-defined curvature $2$-form on $B^3$ (see for example \cite{KoNo1, KoNo2}), thereby simplifying our definition. This is the reason why our results do not immediately generalize to the nonabelian case.\\
 Up to a universal constant, we can suppose that the projected $2$-forms have integral in $\mathbb Z$ along any closed surface (possibly containing singular points).
\end{rmk}
We have therefore the following candidate for the class seeked in Open Problem \ref{op}:
\begin{definition}[$L^p$-curvatures of singular $U(1)$-bundles \cite{RivKess}, Definition II.1]\label{defcurv}
An $L^p$-curvature of a singular $U(1)$-bundle over $B^3$ is a measurable real-valued $2$-form $F$ satisfying
\begin{itemize}
\item 
$$
\int_{B^3}|F|^pdx^3 <\infty,
$$
\item For all $x\in B^3$ and for almost all $0<r<\op{dist}(x, \de B^3)$ we have 
$$
\int_{\de B_r(x)}i^*_{\de B_r(x)}F\in\mathbb Z,
$$
where $i^*_{\de B_r(x)}$ is the inclusion map of $\de B_r(x)$ in $B^3$.
\end{itemize}
We call $\mathcal F_{\mathbb Z}^p(B^3)$ the class of all such $2$-forms $F$.
\end{definition}
We observe that the above class is clearly closed in the \emph{strong} $L^p$-topology:
\begin{lemma}\label{clostrong}
 The class $\mathcal F_{\mathbb Z}^p(B^3)$ is closed for the $L^p$ topology.
\end{lemma}
\begin{proof}
We take a sequence $F_k\in\mathcal F^p_{\mathbb Z}(B^3)$ such that $F_k\stackrel{L^p}{\to}F_\infty$. If we take $x\in B^3, R<\op{dist}(x,\de B^3)$, then there holds
$$
||F_k - F_\infty||_{L^p}^p\geq \int_{B_R(x)}|F_k - F_\infty|^pdx\geq\int_0^R\left|\int_{\de B_r(x)}i^*_{\de B_r(x)} (F_k- F_\infty)d\mathcal H^2\right|^pdr.
$$
Therefore the above $L^p$-functions 
$$
f_k:[0,R]\to\mathbb Z,\;f_k(r):=\int_{\de B_r(x)}i^*_{\de B_r(x)} F_kd\mathcal H^2
$$
converge to the analogously defined function $f_\infty$ in $L^p$, therefore also pointwise almost everywhere, thus proving that $F_\infty$ also satisfies the properties in Definition \ref{defcurv}.
\end{proof}
\begin{rmk}
 It has been already proved in \cite{Kessel, RivKess} that the class $\mathcal R^\infty$ is dense in $F^p_{\mathbb Z}$ for the $L^p$-topology (see also \cite{P1}).
\end{rmk}
 The fact that $c_1(P|_{\de B_\e(x)})\neq 0$ for all $\e$ small enough implies that the curvature is not in $L^p$ for $p\geq 3/2$ (for example the form $F=(4\pi r^2)^{-1}d\theta\wedge d\phi\in \Omega^2(B^3\setminus\{0\})$ represents the curvature of an $U(1)$-bundle over $B^3\setminus\{0\}$ having $c_1=1$ on all spheres containing the origin, and it is an easy computation to show that $F\notin L^{3/2}$. Therefore the study of the above defined $YM$ functional (i.e. the one equal to the $L^2$-norm of the curvature) is trivial, no topological charge being possible. Therefore we study a similar functional where we substitute the $L^p$-norm to the $L^2$-norm:
$$
YM_p(P)=\int_M|F_A|^p dx^3.
$$
From the above discussion it follows that singularities realizing nontrivial first Chern numbers arise only if $p<3/2$. For such $p$ the class $\mathcal F^p_{\mathbb Z}(B^3)$ is in bijection with the class $L^p_{\mathbb Z}(B^3,\mathbb R^3)$ present in Theorem \ref{mainthm1}, via the identification of $k$-covectors $\beta$ with $(n-k)$-vectors $*\beta$ in $\mathbb R^n$, given by imposing 
\begin{equation}\label{veccovec}
\langle\alpha, *\beta\rangle=\langle\alpha\wedge\beta,\vec e\rangle
\end{equation}
for all $(n-k)$-covectors $\alpha$, where $\vec e$ is an orientating vectorfield of $\mathbb R^n$. After this identification, we can reformulate Theorem \ref{mainthm1} as follows:
\begin{theorem}[Main Theorem]\label{mainthm}
 If $p>1$ and $F_n\in\mathcal F_{\mathbb Z}^p(B^3)$ and $||F_n||_{L^p}\leq C<\infty$ then we can find a subsequence $F_{n'}$ converging weakly in $L^p$ to a $2$-form in $\mathcal F_{\mathbb Z}^p(B^3)$.
\end{theorem}
This answers Open Problem \ref{op} in the case of $U(1)$-bundles:
\begin{corollary}[Solution to Open Problem \ref{op} in the case of $U(1)$-bundles]\label{solop}
 In the case of $U(1)$-bundles, the class $\mathcal F_{\mathbb Z}^p(B^3)$ with the sequential weak $L^p$-topology solves Open Problem \ref{op} when $1<p<3/2$.
\end{corollary}
A direct consequence of the above two results is the existence of minimizing $U(1)$-curvatures in $\mathcal F_\mathbb Z^p(B^3)$ under extra constraints, such as for example the imposition of a nontrivial boundary datum\footnote{We observe that defining the Dirichlet boundary value minimization problem for $YM_p$ on $\mathcal F_{\mathbb Z}^p (B^3)$ is a delicate issue, which will therefore be treated separately (see \cite{P1}).}.

\subsection{Main points of the proof and outline of the paper}
The counterexample in Proposition \ref{counterex} shows that Theorem \ref{mainthm} cannot hold in case $p=1$, as in such case we obtain all currents and we possibly loose any integrality condition by weak convergence. This means that the convexity of the $L^p$-norm arising when $p>1$ is really needed for a result similar to Theorem \ref{mainthm} to hold. On the other hand, the notions of a minimal connection as in \cite{BCL} or in \cite{Iso2} are based on the duality between currents and smooth functions, where again no convexity is involved. Therefore we face the difficulty of finding a strategy more adapted to our problem. A difficulty arising from the presence of a $L^p$-exponent different from $1$ arised also in the work of Hardt and Rivi\`ere \cite{HR}, where an extension of the the Cartesian Currents notion of a minimal connection had to be introduced in order to treat singularities of functions in $W^{1,3}(B^4, S^2)$. For such definition one had to consider the class of \emph{scans}, which are, roughly, a generalization of currents where the mass of slices is taken in $L^\alpha$-norm with $\alpha<1$ (instead of $\alpha=1$, which would give back the usual mass as in \cite{Federer}, 4.3). In order to achieve the weak compactness result analogous to our Theorem \ref{mainthm}, a particular distance between scans was introduced, which allowed a $L^{1/\alpha,\infty}$-estimate. Such procedure was inspired by the approach of Ambrosio and Kirchheim \cite{AK}, Sections 7 and 8, which used $BV$ (instead of $L^{1/\alpha,\infty}$) bounds for functions with values in a suitable metric space, obtaining a rectifiability criterion and a new proof of the closure theorem for integral currents, via a maximal function estimate. \\

In the case of \cite{AK} the metric space considered was the one of rectifiable currents arising as slices of an initial current, with the flat metric. In \cite{HR} a distance $d_e$ extending the definition of the flat metric was considered on the space of scans arising as slices of graphs. In our case we introduce a metric on the space $Y$ of $L^p$-forms arising as slices on concentric spheres of a given curvature $F\in\mathcal F_{\mathbb Z}^p$:
$$
Y:=L^p(S^2)\cap\left\{h:\,\int_{S^2}h\in\mathbb Z\right\}.
$$
In our case, for $h_1, h_2\in Y$ we define
$$
d(h_1, h_2):=\inf\left\{\|X\|_{L^p}:\: h_1-h_2 = \op{div}X + \de I + \sum_{i=1}^N d_i\,\delta_{a_i}\right\},
$$
where the infimum is taken over all triples given by a $L^p$-vectorfield $X$, an integer $1$-current $I$ of finite mass, and a finite set of integer degree singularities, given by an $N$-ple of couples $(a_i,d_i)$, where $a_i\in S^2$ and $d_i\in\mathbb Z$.\\

The fact that $d$ is a metric is not immediate (see \textit{Section \ref{sec:defmetric}}): in particular the implication
$$
d(h_1,h_2)=0\Rightarrow h_1=h_2
$$
depends upon a result (see \cite{P}, and Proposition \ref{petrache}) which says that flow lines of a $L^p$-vectorfield on $S^2$ with $\op{div}V=\de I$ where $I$ is an integer multiplicity rectifiable $1$-current of finite mass can be represented as preimages $u^{-1}(y),\;y\in\mathbb S^1$, for some $u\in W^{1,p}(S^2, S^1)$.\\

The estimate connecting the distance $d$ above to the ideas of \cite{AK, HR} is (see Proposition \ref{maximalest}) a bound on the lipschitz constant of the slice function 
$$
h:[s', s]\to (Y,d),\quad x\mapsto h(x):=T_x^*i_{\de B_x(a)}^*F,
$$
where $T_x(\theta):=a+x\theta$ maps $S^2$ to $\de B_x(a)$. We estimate the lipschitz constant of $h$ in terms of the maximal function of the $L^1$-function 
$$
f:[s',s]\to\mathbb R^+,\quad f(x):=\|h(x)\|_{L^p}^p,
$$
an estimate in the same spirit of the one used in \cite{HR}, which was a generalization of the pioneering approach of \cite{AK}.\\

In \textit{Section \ref{sec:applicmetric}} we prove a modified version of Theorem 9.1 of \cite{HR}, which from the uniform $L^{p,\infty}$-bound on a sequence of maximal functions $Mf_n$ defined as above (which is a direct consequence of the uniform $L^p$-bound on the sequence of curvatures $F_n$ considered initially), allows us to deduce a kind of \emph{locally uniform} pointwise convergence of the slices $h_n(x)$ for a.e. $x$, up to the extraction of a subsequence. This uniformity is the main advantage of our whole construction, and this is why we have to introduce the above distance and maximal estimate. The seed from which our technique grew was planted by \cite{AK}, and first developed in \cite{HR}.\\

\textit{Section \ref{sec:verifhyphr}} is devoted to the verification of the hypotheses of the abstract Theorem \ref{hr}, and \textit{Section \ref{sec:proofmainthm}} concludes that we can extract a subsequence as requested by Theorem \ref{mainthm}.\\

The last \textit{Section \ref{sec:remarks}} is devoted to the proving the ``$p=1$'' part of Theorem \ref{mainthm1}, thereby also justifying the assumption ``$p>1$'' of Theorem \ref{mainthm}.

\section{Definition of the metric}\label{sec:defmetric}
We consider the following function on $2$-forms\footnote{In order to avoid heavy notations, we will often use the formula \eqref{veccovec}, identifying $k$-vectors with $(n-k)$-differential forms in an $n$-dimensional domain, without explicit mention.} in $L^p(\wedge^2D)$, for some smooth domain $D\subset\R{2}$ or for $D=S^2$, and for $1<p<3/2$:
\begin{equation*}
d(h_1,h_2)=\inf\left\{||X||_{L^p}:\:h_1-h_2=\op{div}X+\de I+\sum_{i=1}^Nd_i\delta_{a_i}\right\},
\end{equation*}
where the infimum is taken over all triples given by a $L^p$-vectorfield $X$, an integer $1$-current $I$ of finite mass, and a finite set of integer degree singularities, given by an $N$-ple of couples $(a_i,d_i)$, where $a_i$ are points in $D$ and the numbers $d_i\in\mathbb Z$ represent the topological degrees of $X$ near the singularities $\{a_i\}$.
\begin{rmk}
 We observe that up to changing the current $I$ and the singularity set in the above triples, we may reduce to considering the case where the singularity set contains only one given point, for example the origin $0$. More precisely, we can say that an equivalent formulation of the above distance is
\begin{equation*}
d(h_1,h_2)=\inf\left\{||X||_{L^p}:\:h_1-h_2=\op{div}X+\de I+d\delta_0\right\},
\end{equation*}
where the infimum is now taken on all triples $(X,I,d)$, where $X,I$ are as above, and $d$ is some integer number. 
\end{rmk}
\begin{rmk}
 In particular, from the above it follows that 
\begin{equation}\label{int}
 d(h_1,h_2)\neq\infty\text{ implies }\int_D(h_1-h_2)\in\mathbb Z.
\end{equation}
Therefore the following function $\tilde d$ is a priori different than $d$ (and can be seen as an extension of $d$)
\begin{equation}\label{newd}
 \tilde d(h_1,h_2)=\inf\left\{||X||_{L^p}:\:h_1-h_2=\op{div}X+\de I+\delta_0\int_D(h_1-h_2)\right\},
\end{equation}
since there is no apparent reason for it to be infinite when $\int_D(h_1-h_2)\notin\mathbb Z$.
\end{rmk}
\begin{proposition}\label{dismetric}
 The above defined function $d$ is a metric on $L^p(\wedge^2D)$, both in the case when $D=[0,1]^2$ and in the case $D=S^2$.
\end{proposition}
\begin{proof}
We will prove the three characterizing properties of a metric.
\begin{itemize}
\item \emph{Reflexivity:} This is clear since the $L^p$-norm, the space of integer $1$-currents of finite mass and the space of finite sums $\sum_{i=1}^Nd_i\delta_{a_i}$ as above, are invariant under sign change.
\item \emph{Transitivity:} If we can write
\begin{equation*}
 \left\{
\begin{array}{l}
h_1-h_2=\op{div}X_\e +\de I_\e+\sum_{i=1}^Nd_i\delta_{a_i}\\
h_2-h_3=\op{div}Y_\e +\de J_\e+\sum_{j=1}^Me_j\delta_{b_j},
\end{array}
\right.
\end{equation*}
where
\begin{equation*}
 \left\{
\begin{array}{l}
||X_\e||_{L^p}\leq d(h_1,h_2)+\e\\
||Y_\e||_{L^p}\leq d(h_2,h_3)+\e,
\end{array}
\right.
\end{equation*}
then we put $Z_\e:=X_\e+Y_\e$, $K_\e=I_\e+J_\e$ and we consider the singularity set $\{(c_k,f_k)\}$ where
\begin{eqnarray*}
 \{c_k\}&=&\{a_i\}\cup\{b_j\}\\
f_k&=&\left\{\begin{array}{lll} 
              d_i&\text{ if }c_k=a_i,\;c_k\notin\{b_j\}\\
              e_j&\text{ if }c_k=b_j,\;c_k\notin\{a_i\}\\
              d_i+e_j&\text{ if }c_k=a_i=b_j.
             \end{array}\right.
\end{eqnarray*}
We see that $K_\e$ is still an integer $1$-current of finite mass and that $h_1-h_3=\op{div}Z_\e+\de K_\e+\sum_k f_k\delta_{c_k}$. Then we have:
\begin{equation*}
 d(h_1,h_3)\leq||Z_\e||_{L^p}\leq||X_\e||_{L^p}+||Y_\e||_{L^p}\leq d(h_1,h_2)+d(h_2,h_3)+2\e,
\end{equation*}
and as $\e\to 0$ we obtain the transitivity property of $d(\cdot,\cdot)$.
\item \emph{Nondegeneracy:} This is the statement of the following proposition.
\end{itemize}
\end{proof}
\begin{proposition}\label{nondegen}
 Under the hypotheses above, $d(h_1,h_2)=0$ implies $h_1=h_2$ almost everywhere, for $1<p<2$.
\end{proposition}
\begin{proof}
We may suppose without loss of generality that $\int_D(h_1-h_2)\in\mathbb Z$.\\
We start by taking a sequence of forms $X_\e$ such that 
\begin{equation*}
\left\{
\begin{array}{l}
||X_\e||_{L^p}\to 0\\
h_1-h_2=\op{div}X_\e+\de I_\e+\delta_0\int_D(h_1-h_2).
\end{array}
\right.
\end{equation*}
We would be almost done, if we could control also the convergence of the $1$-currents $I_\e$. To do so, we start by expressing the boundaries $\de I_\e$ in divergence form. Therefore, we consider the equations
\begin{equation}\label{h1h2}
\left\{
\begin{array}{l}
\Delta\psi=h_1-h_2+\delta_0\int_D(h_2-h_1)\\
\int_D\psi=0
\end{array}
\right.
\end{equation}
(by classical results, this equation has a solution whose gradient is in $L^q$ for all $q$ such that $q<2$ and $q\leq p$) and
\begin{equation}\label{xe}
\left\{
\begin{array}{l}
\Delta\varphi_\e=\op{div}X_\e\\
\int_D\varphi_\e=0
\end{array}
\right.
\end{equation}
This second equation can be interpreted in terms of the Hodge decomposition of the $1$-form associated to $X_\e$: indeed, for a $L^p$ $1$-form $\alpha$ we know by classical results that it can be Hodge-decomposed as 
\begin{equation*}
\left\{
\begin{array}{l}
\alpha=df+d^*\omega+h\text{, where}\\
\int_Df=0,\;\int_D*\omega=0,\;\Delta h=0,\text{ and}\\
||df||_{L^p}+||d^*\omega||_{L^p}+||h||_{L^p}\leq C_p||\alpha||_{L^p}.
\end{array}
\right.
\end{equation*}
Therefore in equation \eqref{xe} we can associate (via formula \eqref{veccovec}) a $1$-form $\alpha$ to $X_\e$ and take $\varphi_\e$ equal to the function $f$ coming from the above decomposition. Then an easy verification shows that \eqref{xe} is verified.\\

We have thus, that both \eqref{h1h2} and \eqref{xe} have a solution, and such solutions satisfy the following estimates:
\begin{equation*}
\left\{
\begin{array}{l}
||\nabla\varphi_\e||_{L^p}\leq c_p||X_\e||_{L^p}\to 0\\
\nabla\psi\in W^{1,p}\subset L^p\text{ since }p^*=\frac{2p}{2-p}>p.
\end{array}
\right.
\end{equation*}
Then (supposing $p<2$) we obtain
\begin{equation}\label{estimate1}
\left\{
\begin{array}{l}
\de I_\e=\op{div}(\nabla(\varphi_\e - \psi))\\
||\nabla(\varphi_\e - \psi)||_{L^p}\text{ is bounded}
\end{array}
\right.
\end{equation}
Now we consider the vector field $\nabla(\varphi_\e - \psi):=V_\e\in L^p(D,\R{2})$.
\begin{proposition}[\cite{P}]\label{petrache}
 Suppose that we have a function $V\in L^p(D,\R{2})$ with $p>1$, for a domain $D\subset\R{2}$ or for $D=S^2$, whose divergence can be represented by the boundary of an integer $1$-current $I$ on $D$, i.e. for all test functions $\gamma\in C^{\infty}_c(D,\R{})$ we have
\begin{equation}\label{integerdiv}
\int_D \nabla\gamma(x)\cdot V(x)dx=\langle I, \nabla\gamma\rangle.
\end{equation}
Then there exists a  $W^{1,p}$-function $u:D\to S^1\simeq \mathbb R/2\pi\mathbb Z$ such that $\nabla^\perp u=V$.
\end{proposition}
Applying Lemma \ref{integerdiv} to the current $I_\e$ of \eqref{estimate1}, we can write
\begin{equation*}
\left\{
\begin{array}{l}
\nabla^\perp u_\e=\nabla(\varphi_\e-\psi)\\
\de I_\e=\op{div}(\nabla(\varphi_\e-\psi))\\
||\nabla u_\e||_{L^p}\leq C||\nabla(\varphi_\e-\psi)||_{L^p}\leq C.
\end{array}
\right.
\end{equation*}
Then we have that a subsequence $u_k$ of the $u_\e$ converges weakly in $W^{1,p}(D,\R{2})$ to a limit $u_0$, and thus it converges in $L^1_{loc}$, proving that $u_0\in W^{1,p}(D,\R{2})$. Now, the $u_k$ converge to $u_0$ almost everywhere, and thus the limit function $u_0$ also has almost everywhere values in $S^1$. Since we know now that $u_k\stackrel{L^1}{\to}u_0$ and that $||u_k||_{L^\infty}\leq 1$, we obtain by interpolation $u_k\stackrel{L^r}{\to}u_0$ for all $r<\infty$. Therefore (by choosing $r=\tfrac{q}{q-1}$ and by Young's inequality) it follows that 
\begin{equation}\label{convergence}
\nabla^\perp u_k\stackrel{L^1}{\to}\nabla^\perp u_0.
\end{equation}
By a generalization of Sard's theorem, the fibers $F_\e(\sigma):=\{x\in D:\;u_\e(x)=\sigma\}$ for $\sigma\in S^1$ are rectifiable for almost all $\sigma$ and can be given a structure of integer $1$-currents. Then for almost all $\sigma\in S^1$ we have
\begin{equation*}
 \de\left[F_\e(\sigma)\right]=\de I_\e.
\end{equation*}
By \eqref{convergence} we also obtain that the $L^p$-weak limit $\nabla(\varphi_0-\psi)$ exists up to extracting a further subsequence, and it is equal to $\nabla^\perp u_0$. Therefore, again by Sard's theorem, its divergence is the boundary of an integer $1$-current $I_0$, which can be described using a generic fiber $F_0(\sigma)$ of $u_0$:
\begin{equation*}
 \op{div}\nabla(\varphi_0 - \psi)=\de I-0.
\end{equation*}
Since $u_0\in W^{1,p}$, by an easy application of the Fubini theorem to the generalized coarea formula, we have that the generic fibers $F(\sigma)$ have finite $\mathcal H^1$-measure, thus $I_0$ has finite mass. \\

Since $\nabla\psi\in L^p$, from
$$
\nabla^\perp u_k=\nabla(\psi-\varphi_k)\stackrel{L^1}{\to}\nabla^\perp u_0
$$
we deduce that $\nabla\varphi_k\stackrel{L^1}{\to}\nabla\varphi_0$. On the other hand, $\nabla\varphi_\e\stackrel{L^p}{\to}0$ together with \eqref{h1h2}, implies that there exists an integer $1$-current such that
\begin{equation}\label{lpcurrent}
 h_1-h_2=\de I_0.
\end{equation}
The following lemma concludes the proof.
\end{proof}
\begin{lemma}\label{boundarynotlp}
 If the boundary of an integer multiplicity finite-mass $1$-current $I$ on a domain $D\subset\R{2}$ can be represented by a $L^p$-function for $p\geq 1$, then $\de I=0$.
\end{lemma}
\begin{proof}
 Suppose for a moment that $\de I\neq 0$ and that there exists a function $h$ such that for all $\varphi\in C^1_c(D)$ there holds 
\begin{equation*}
 \langle \varphi, h\rangle = \langle\varphi,\de I\rangle.
\end{equation*}
If we take a smooth positive radial function $\varphi\in C^1_c(B_1(0))$ which is equal to $1$ on $B_{1/2}(0)$ and we consider a point of approximate continuity $x_0$ of $h$ such that $h(x_0)\neq 0$, then we will also have 
\begin{eqnarray*}
 \mathbb M(I)||\nabla\varphi||_{L^\infty}&\geq&\left|\left\langle\nabla_x\left(\frac{1}{\e}\varphi(\e( x-x_0))\right),I\right\rangle\right|\\
&=&\frac{1}{\e}\left|\int\varphi(\e (x-x_0)) h(x)\;dx\right|\\
&\geq& \frac{c\;|h(x_0)|}{\e},
\end{eqnarray*}
which for $\e>0$ small enough is a contraddiction.
\end{proof}

\section{Application of the above defined metric in the case of $D=S^2$}\label{sec:applicmetric}
We hereby consider a $2$-form $h$ on $B^3:=B^3_1(0)$ such that $i^*_{\de B^3}h=0$ and we suppose that for a fixed point $a\in B^3$ and for $0<s'<s<\op{dist}(a,\de B^3)$ there holds
\begin{equation*}
 \forall x\in[s',s],\;\int_{\de B_x(a)} i^*_{\de B_x(p)}h\in\mathbb Z.
\end{equation*}
We also suppose that there exists an integral $1$-current $I$ in $[0,1]^3$ such that $\de I$ can be represented by $*dh$. In this case we have the following result:
\begin{proposition}\label{maximalest}
 Under the above hypotheses, for each subinterval $K\subset[s',s]$ there exists a function $M_K\in L^{1,\infty}(K,\R{})$, such that there holds 
\begin{equation}\label{maximalf}
 \left[M_K(x)\right]^{1/p}\geq\op{esssup}_{x\neq\tilde x\in K}\frac{d(h(x),h(\tilde x))}{|x-\tilde x|},
\end{equation}
Where the $2$-form $h(x):=T_x^*i_{\de B_x(a)}^*h$ on $S^2$ corresponds to the restriction $i_{\de B_x(a)}^*h$ through the affine map $T_x:S^2\to \de B_x(a)$, $T_x(\theta):=a+x\theta$.
\end{proposition}
\begin{proof}
 Without loss of generality, we may suppose that $s=1$ and that $a$ is the origin. We start by observing that given a subinterval $K'=[t,t+\delta]\subset K$, we may consider (in polar coordinates) a function $\bar\varphi(\theta,r)=\varphi(\theta)$ on $B_1(0)\setminus\{0\}$ and identify the $2$-form $h$ with the $1$-form $*h$. Then for $x\in]0,1]$, $i_{\de B_x(0)}^*h$ will be identified with a $1$-form tangent to $\de B_x(0)$, and therefore $h(x)$ is identified with a $1$-form (or, after fixing the standard metric, with a $1$-vector field) on $S^2$. Observe that 
\begin{eqnarray*}
\left\langle \varphi,*_{S^2}d_{S^2} \left( \int_t^{t+\delta}h(x)dx \right) \right\rangle_{S^2}
&=&\int_{S^2}\nabla\varphi(\theta)\cdot \left(\int_t^{t+\delta}h(x)(\theta)\;dx\right)\;d\theta\\
&=&\int_t^{t+\delta}\int_{S^2}\langle d\varphi(\theta),h(x)(\theta)\rangle\;dx\;d\theta\\
&=&\int_\Omega\langle d\varphi(\theta), i_{\de B_x}^*(*h)(\theta)\rangle\;dV\\
&=&\int_\Omega \langle d\bar\varphi,*h\rangle \;dV\\
&=&\int_\Omega \langle \bar\varphi,*dh\rangle \;dV + \int_{\de B_{t+\delta}}*(i_{\de B_{t+\delta}}^*h)\bar\varphi\;d\sigma \\
&&- \int_{\de B_{t}}*(i_{\de B_{t}}^*h)\bar\varphi\;d\sigma \\
&=& \int_\Omega \langle \bar\varphi,*dh\rangle \;dV + \int_{S^2}h(t+\delta)\varphi \;d\theta \\
&&- \int_{S^2}h(t)\varphi\;d\theta,
\end{eqnarray*}
where $\Omega:=B_1\setminus B_{s'}$. We used above the definition of $h(x)$ and the fact that since $\bar\varphi$ depends only on $\theta$ we have that for any one-form $\omega$ there holds $\langle d\bar\varphi,\omega\rangle_{\de B_x}= \langle d\bar\varphi,i_{\de B_x}^*\omega\rangle_{\de B_x}$. We now use the property relating the $1$-current $I$ to the form $h$:
\begin{equation*}
 \int_\Omega \langle \bar\varphi,*dh\rangle \;dV =\left\langle\bar\varphi,(\de I)\llcorner\Omega\right\rangle.
\end{equation*}
 The following formula holds for $C^1$-approximations $\chi_\e\in C^\infty_c(]0,1[^3)$ of the characteristic function of $\Omega$:
\begin{equation*}
 (\de I)\llcorner\Omega=\lim_{\e\to 0}(\de I)\llcorner \chi_\e=\lim_{\e\to 0}\left[\de(I\llcorner \chi_\e)+I\llcorner(d\chi_\e)\right]=\de(I\llcorner\Omega)+\lim_{\e\to 0}I\llcorner(d\chi_\e),
\end{equation*}
and the last term can be expressed in terms of slices along the proper function 
\begin{eqnarray*}
f:B_1\setminus B_{s'}&\to&[s',1]\\
(\theta, r)&\mapsto&r, 
\end{eqnarray*}
 keeping in mind that $\Omega=f^{-1}([t,t+\delta])$: we have
\begin{equation*}
 \lim_{\e\to 0}I\llcorner(d\chi_\e)=\langle I,f,t+\delta\rangle - \langle I,f,t\rangle,
\end{equation*}
and we observe therefore that for almost all values of $t$ and $t+\delta$ the above contribution is an integer $0$-current, so from 
\begin{equation*}
 \int_{s'}^1\mathbb M\langle I,f,\tau\rangle d\tau=\mathbb M\left( I\llcorner f^\#(\chi_{[s',1]}d\tau)\right)\leq C_{s'}\mathbb M(I)<\infty,
\end{equation*}
we obtain that it has also finite mass for almost all choices of $t$ and $t+\delta$, therefore it is a finite sum of Dirac masses with integer coefficients. We now use the following easy lemma:
\begin{lemma}\label{integerprojection}
 With the above notations, if $\bar J$ is a finite mass rectifiable integer $1$-current in $B_y\setminus B_x$ for $1>y>x>0$, then there exists a finite mass rectifiable integer $1$-current supported on $\de B_x$ such that 
\begin{itemize}
 \item for all functions $\bar\varphi(\theta,r)=\varphi(\theta)\chi(r)$ where $\chi\in C^\infty_c(]0,1])$ and $\chi\equiv 1$ on $[x,y]$, there holds $\langle \bar\varphi,\de\bar J\rangle=\langle\varphi,\de J\rangle$,
 \item $\mathbb M(\bar J)\leq\mathbb M(J)$
\end{itemize}
\end{lemma}
Applying the above lemma to $\bar J=I\llcorner\Omega$, we obtain 
\begin{equation*}
 \langle\de(I\llcorner\Omega),\bar\varphi\rangle=\langle \de J,\varphi\rangle,
\end{equation*}
where $J$ is a finite mass rectifiable integer $1$-current. We can summarize what shown so far by writing (all the objects being defined on $S^2$)
\begin{eqnarray*}
 *d \left( \int_t^{t+\delta}h(x)dx \right)&=&h(t+\delta)-h(t) +\langle I,f,t+\delta\rangle - \langle I,f,t\rangle+\de J\\
&=&h(t+\delta)-h(t) +\sum_{i=1}^Nd_i\delta_{a_i} +\de J.
\end{eqnarray*}
Therefore, by definition of the metric $d(\cdot,\cdot)$, it follows that
\begin{equation*}
 d(h(t),h(t+\delta))\leq\left\|\int_t^{t+\delta}h(x)\:dx\right\|_{L^p(S^2)}.
\end{equation*}
We further compute:
\begin{eqnarray*}
 d(h(t),h(t+\delta))&\leq&\left[\int_{S^2}\left|\int_t^{t+\delta}h(r)(\theta)\:dr\right|^p\;d\theta\right]^{1/p}\\
&\leq&\delta^{1-\frac{1}{p}}\left[\int_t^{t+\delta}\int_{S^2}|h(r)(\theta)|^p\;dr\;d\theta\right]^{1/p}\\
&\leq&\delta \left[ M_K\left( \int_{S^2}|h(\cdot)|^p \right)(t) \right]^{1/p},
\end{eqnarray*}
where $M_Kf$ is the uncentered maximal function of $f$ on the interval $K$, defined as
\begin{equation*}
M_Kf(x)=\sup\left\{\tfrac{1}{|B_\rho(Y)|}\int_{B_\rho(Y)}|f|:\:x\in B_\rho(y)\subset K\right\}.
\end{equation*}
\end{proof}

\section{The almost everywhere pointwise convergence theorem}\label{sec:aeconvthm}
We next call 
\begin{equation*}
 N_Ih(t):=\left[M_I\left(\|h(r)\|_{L^p(D)}^p\right)(t)\right]^{1/p},
\end{equation*}
where $D=[0,1]^2$ or $D=S^2$.\\

Then the following is a restatement of the equation \eqref{maximalf} in terms of $N_kh$:
\begin{equation}\label{lipschitzconst}
 \text{For all }x,y\in I\text{, there holds }N_Ih(x)|x-y|\geq d(h(x), h(y)).
\end{equation}
Consider now the metric space 
\begin{equation}\label{spacey}
 Y:=\left[L^p(D), d(\cdot,\cdot)\right]\cap\{h:\int_Dh\in \mathbb Z\}.
\end{equation}
It is clear that $f:=\left[t\mapsto\|h(t)\|_{L^p(D)}^p\right]\in L^1([s',s])$ for all $0<s'<s\leq 1$, therefore, by the usual Vitali covering argument for $M_If$ we obtain that there exists a dimensional constant $C$ for which
\begin{equation}\label{lpinfty}
 \sup_{\lambda>0}\lambda^p|\{t\in I:\;N_Ih(t)>\lambda\}\leq C\int_I|f(x)|dx.
\end{equation}
We can now prove the following analogue of \cite{HR}'s Theorem 9.1 (a proof is provided just in order to convince the reader that the hypotheses in the original statement can be changed: in fact it is completely analogous to the original one).
\begin{theorem}\label{hr}
 Suppose that for each $n=1,2,\ldots, h_n:[0,1]\to Y$ is a measurable function such that for all sbintervals $I\subset [0,1]$ there holds 
\begin{equation}\label{weakestimate}
\sup_{\lambda>0}\lambda^p|\{t\in I\;N_Ih_n(t)>\lambda\}|\leq\mu_n(I)
\end{equation}
for some function $N_Ih_n$ satisfying \eqref{lipschitzconst}, where $\mu_n$ are positive measures on $[0,1]$ such that $\sup_n\mu_n([0,1])< \infty$. We also suppose that a lower semicontinuous functional $\mathcal N:Y\to \R{+}$ is given, and that
\begin{itemize}
 \item the sublevels of $\mathcal N$ are sequentially compact
\item  there holds
\begin{equation}\label{conditionn}
 \sup_n\int_{[0,1]}\mathcal N(h_n(x))dx<L<\infty\text{ for some }L\in\R{}.
\end{equation}
\end{itemize}
Then the sequence $h_n$ has a subsequence that converges pointwise almost everywhere to a limiting function $h:X\to Y$ satisfying
\begin{itemize}
 \item $\int_{[0,1]}\mathcal N(h(x))dx\leq L$,
 \item $\forall I\subset [0,1],\;\sup_{\lambda>0}\lambda^p|\{t\in I\;\tilde N_Ih(t)>\lambda\}\leq\sup_n\mu_n(I)$, where again $\tilde N_Ih$ satisfies \eqref{lipschitzconst}.
\end{itemize}
\end{theorem}
\begin{rmk}\label{rmkintervals}
 In Theorem \ref{hr} we considered the interval $[0,1]$ instead of $[s',s]$ just for the sake of simplicity; the above results clearly extend also to the general case.
\end{rmk}
\begin{proof}
\textbf{Claim 1.} \textit{It is enough to find a subsequence $f_{n'}$ which is pointwise a.e. Cauchy convergent.} Indeed, in such case for a.e. $x\in[0,1]$ there will exist a unique limit $f(x):=\lim f_{n'}(x)\in\hat Y$, the completion of $Y$. For such $x$ we can then use Fatou's lemma and \eqref{conditionn}, obtaining for a.e. $x$ a further subsequence $n''$ (which depends on $x$), along which $\mathcal N(f_{n''}(x))$ stays bounded. By compactness of the sublevels of $\mathcal N$ we then have that $f(x)\in Y$.\\

Next, the lower semicontinuity of $\mathcal N$ implies that the property \eqref{conditionn} passes to the limit, while for the other claimed property we may take
\begin{equation*}
\tilde N_Ih(x):=\sup_{I\ni\tilde x\neq x}\frac{d(h(x),h(\tilde x))}{|x-\tilde x|}
\end{equation*}
and then use \eqref{lipschitzconst} to obtain
\begin{equation*}
 d(h(x), h(\tilde x))=\lim_{n'} d(h_{n'}(x),h_{n'}(\tilde x))\leq\liminf_{n'}N_Ih_{n'}(x)|x-x'|,
\end{equation*}
which gives \eqref{weakestimate} for $\tilde N_Ih_{n'}$, since it shows that $\tilde N_Ih(x)\leq\liminf_{n'}N_Ih_{n'}(x)$. This proves Claim 1.\\

\textbf{Wanted properties.} We will obtain the wanted subsequence $(n')$ by starting with $n_0(j)=j$ and successively extracting a subsequence $n_k(j)$ of $n_{k-1}(j)$ for increasing $k$. In parallel to this (for each $k\geq 0$)
\begin{itemize}
\item we will select countable families $\mathcal I_k$ of closed subintervals of $[0,1]$ which cover $[0,1]$ up to a nullset $Z_k$
\item for $I\in\mathcal I_k$ we will give a point $c_I\in I$ such that $y_{j,I}:=h_{n_k(j)}(c_I)$ are Cauchy sequences for all $I\in\mathcal I_k$ and 
\begin{equation}\label{condition1}
 \limsup_jN_Ih_{n_k(j)}(c_I)\leq\frac{1}{k|I|}
\end{equation}
\end{itemize}
\textbf{Claim 2.} \textit{The above choices guarantee the existence of a pointwise almost everywhere Cauchy subsequence $h_{n'}$.} Indeed, we can then take a diagonal subsequence $j'=n_j(j)$, and use the fact that the nullsets $Z_k$ have as union a nullset $Z$. Then for $I\in\mathcal I_k$ with $k$ big enough, we have $d(f_{i'}(c_I),f_{j'}(c_I))<\e/3$ for $i', j'$ big enough, while for $x\in I$, by \eqref{condition1} there exists $C$ close to $1$ such that
\begin{equation*}
d(h_{i'}(x), h_{i'}(c_I))\leq N_Ih_{i'}(c_I)|I|\leq C\frac{1}{k}.
\end{equation*}
From these two estimates it follows that for all $x\in [0,1]\setminus Z$ the sequence $h_{j'}$ is Cauchy, as wanted.\\

\textbf{Obtaining the wanted properties.} The subsequence $n_k(j)$ of $n_{k-1}(j)$ will be also obtained by a diagonal extraction applied to a nested family of subsequences $n_{k-1}\prec m_1\prec m_2\prec\ldots$ (where $a\prec b$ means that $b(j)$ is a subsequence of $a(j)$). We describe now the procedure used to pass from $n_{k-1}$ to $m_1$.\\
 We choose an integer $q$ such that
\begin{equation*}
 q>2k^p\sup_n\mu_n([0,1]) 
\end{equation*}
and we let $\mathcal I$ be the decomposition of $[0,1]$ into $2q$ non-overlapping subintervals of equal length. Then for each $n$ we can find $q$ ``good'' intervals in $\mathcal I$ having $\mu_n$-measure less than $1/(2k^p)$. The possible choices of such subsets of intervals being finite, we can find one such choice of subintervals $\{I_1,\ldots,I_q\}\subset \mathcal I$ and a subsequence $m_0\succ n_{k-1}$ such that for any of these fixed ``good'' intervals and for any $j\in\mathbb N$, there holds
\begin{equation}\label{choicei}
 \mu_{m_0(j)}(I_i)<\frac{1}{2k^p}.
\end{equation}
 For a fixed interval $I_i$, we now give a name to the set of points where \eqref{condition1} is falsified at step $m_0(j)$:
\begin{equation}\label{choicee}
E_{m_0(j)}:=\left\{x\in I_i:\;N_{I_i}h_{m_0(j)}(x)>\frac{1}{k|I_i|}\right\}.
\end{equation}
Then by \eqref{weakestimate}, \eqref{choicei}, \eqref{choicee} and since $|I_i|\leq |I|=1$, we obtain
\begin{equation*}
 |E_{m_0(j)}|\leq k^p|I_i|^p\mu_{m_0(j)}(I_i)<\frac{1}{2}|I_i|^p\leq \frac{1}{2}|I_i|.
\end{equation*}
for $j$ large enough, and therefore by Fatou lemma we get
\begin{equation*}
 \int_{I_i}\liminf_j\left[\chi_{E_{m_0(j)}}(x) + \frac{|I_i|}{3L}\mathcal N(h_{m_0(j)}(x))\right]dx\leq\frac{1}{2}|I_i|+\frac{|I_i|}{3L}L=\frac{5}{6}|I_i|
\end{equation*}
Therefore we can find $c_{I_i}\in I_i$ and a subsequence $m_1\succ m_0$ so that along $m_1$ we have
\begin{equation*}
 \chi_{E_{m_1(j)}}(c_{I_i}) + \frac{|I_i|}{3L}\mathcal N(h_{m_1(j)}(c_{I_i}))<1,
\end{equation*}
in particular $c_{I_i}\notin E_{m_1(j)}$ for all $j$, and $\mathcal N(f_{m_1(j)}(c_{I_i}))$ is bounded. The latter fact allows us to find a Cauchy subsequence $m_2\succ m_1$, while the former one gives us the wanted property \eqref{condition1} for $I_i$. We can further extract such subsequences in order to obtain the same property for all the ``good'' intervals $I_1,\ldots, I_q$. These intervals cover $1/2$ of the Lebesgue measure of $[0,1]$, so we may continue the argument by an easy exhaustion, covering $[0,1]$ by ``good'' intervals up to a set of measure zero.
\end{proof}

\section{Verification of the properties needed in Theorem \ref{hr}}\label{sec:verifhyphr}
We have seen that the functions $N_Ih_n$ defined in Section \ref{sec:aeconvthm} satisfy the hypotheses \eqref{lipschitzconst} and \eqref{weakestimate}, as follows from \eqref{lpinfty} if we choose 
\begin{equation*}
\mu_n(I):=C\int_I\|h_n(x)\|_{L^p(D^2)}^pdx.
\end{equation*}
In order to use the abstract theorem \ref{hr}, we specify the space 
\begin{equation}\label{defy}
 Y:=\{h\in L^p(D,\wedge^2D):\int_Dh\in\mathbb Z\},
\end{equation}
where $D$ is a $2$-dimensional domain (for example $[0,1]^2$ or $S^2$) and we define the functional $\mathcal N:Y\to\R{+}$ by
\begin{equation}\label{defn}
\mathcal N(h):=\int_D|h|^pdx.
\end{equation}
$Y$ is a metric space with the distance $d$ (this was proved in Proposition \ref{dismetric}). We must now show that $\mathcal N$ satisfies the properties stated in Theorem \ref{hr}, namely that it is sequentially lower semicontinuous and that it has sequentially compact sublevels. The proofs are given in the following two propositions.
\begin{proposition}\label{nlsc}
 Under the notations \eqref{defn} and \eqref{defy}, the functional $\mathcal N:Y\to\R{+}$ is sequentially lower semicontinuous.
\end{proposition}
\begin{proof}
 In other words, we must prove that if $h_n\in Y$ is a sequence such that for some $h_\infty\in Y$ there holds 
\begin{equation}\label{convd}
 d(h_n,h_\infty)\to 0, 
\end{equation}
then we also have 
\begin{equation}\label{eq:nlsc}
 \liminf_{n\to\infty}\mathcal N(h_n)\geq\mathcal N(h_\infty).
\end{equation}
We may suppose that the sequence $\mathcal N(h_n)$ is bounded, i.e. the $h_n$ are bounded in $L^p$. Up to extracting a subsequence we then have 
\begin{equation*}
 h_n\stackrel{L^p}{\rightharpoonup}k_\infty,
\end{equation*}
for some $k_\infty\in L^p$. By taking as a test function $f\equiv 1$, which is in the dual space $L^q$ since $D$ is bounded, we also obtain that $k_\infty\in Y$. Up to extracting a subsequence we may also assume that for all $n$ we have $\int_Dh_n=\int_Dk_\infty\in\mathbb Z$. By the lower semicontinuity of the norm with respect to weak convergence, we have:
\begin{equation*}
  \liminf_{n\to\infty}\mathcal N(h_n)\geq\mathcal N(k_\infty).
\end{equation*}
This implies \eqref{eq:nlsc} if we prove
\begin{equation}\label{h=k}
 h_\infty=k_\infty.
\end{equation}
We now write \eqref{convd} using the definition of $d$: there must exist finite mass integer $1$-currents $I_k$ and vectorfields $X_k$ converging to zero in $L^p$ such that
\begin{equation*}
 h_k-h_\infty=\op{div}X_k +\de I_k + \delta_0\int_D(h_k-h_\infty)=\op{div}X_k +\de I_k.
\end{equation*}
Now we proceed as before, i.e. we define $\psi_k$ and $\varphi_k$ by
\begin{equation*}
 \left\{
\begin{array}{l}
h_k-h_\infty=\Delta\psi_k, \int_D\psi_k= 0\\
\Delta\varphi_k=\op{div}X_k,
\end{array}
\right.
\end{equation*}
so that $\op{div}(\nabla(\psi_k - \varphi_k))=\de I_k$. We also have that $\nabla\varphi_k\to 0$ in $L^p$ and $\nabla\psi_k$ is bounded in $W^{1,p}$, thus up to extracting a subsequence we may assume that
\begin{equation*}
 \nabla \psi_k\stackrel{W^{1,p}}{\rightharpoonup} \nabla\psi_\infty.
\end{equation*}
Now by Proposition \ref{petrache} we can write
\begin{equation*}
 \nabla(\psi_k - \varphi_k) = \nabla^\perp u_k
\end{equation*}
for functions $u_k\in W^{1,p}(D,\mathbb R/2\pi\mathbb Z)$ such that $\|\nabla u_k\|_{L^p}\leq C$. Up to extracting a subsequence we have $\nabla u_k\rightharpoonup \nabla u_\infty$ weakly in $L^p$, thus also in $L^1_{loc}$, and in particular
\begin{equation*}
 \nabla\psi_\infty=\nabla^\perp u_\infty.
\end{equation*}
Since weak-$W^{1,p}$-convergence implies $\mathcal D'$-convergence, we have as in the proof of Proposition \ref{nondegen} that 
\begin{equation*}
 \de I_k\stackrel{\mathcal D'}{\to}\de I_\infty+\op{div}(\nabla^\perp u_\infty)=\op{div}\nabla\psi_\infty,
\end{equation*}
where $I_\infty$ is an integer finite mass $1$-current. By Lemma \ref{boundarynotlp} we have than that $\de I_\infty=0$, which implies that 
\begin{equation*}
 h_k-h_\infty\stackrel{\mathcal D'}{\to}0.
\end{equation*}
Therefore we have \eqref{h=k}, which concludes the proof.
\end{proof}
\begin{proposition}\label{ncpctsublev}
Under the notations \eqref{defn} and \eqref{defy}, and for any $C>0$, the set $\{h\in Y:\:\mathcal N(h)\leq C\}$ is $d$-sequentially compact.
\end{proposition}
\begin{proof}
 We must prove that whenever we have a sequence $h_n$ in $Y$ such that $\|h_n\|_{L^p}$ is bounded, then up to extracting a subsequence we have that for some $k_\infty\in Y$ there holds
\begin{equation}\label{sseqconv}
 d(h_n,k_\infty)\to 0.
\end{equation}
We surely have a subsequence of the $h_n$ which is weakly-$L^p$-convergent to a function $k_\infty\in L^p$. Then, as in the proof of Proposition \ref{nlsc} we have $\int_D k_\infty\in\mathbb Z$ and up to extracting a subsequence we may assume that $\int_D(h_n-k_\infty)=0$ for all $n$. Then we define $\psi_n$ to be the solution of
\begin{equation*}
 \left\{
\begin{array}{l}
\Delta \psi_n=h_n-k_\infty\\
\int_D\psi_n= 0,
\end{array}
\right.
\end{equation*}
and we claim that
\begin{equation}\label{normtozero}
 \|\nabla\psi_n\|_{L^p}\to 0.
\end{equation}
This is enough to conclude, since we can then set $X_n=\nabla\psi_n$ which gives an upper bound of $d(h_n,k_\infty)$ which converges to zero, proving \eqref{sseqconv}.\\

In order to prove \eqref{normtozero} we express
\begin{equation*}
 \nabla\psi_n(x)=\int_D\nabla G(x,y)\left[h_n(y) - k_\infty(y)\right]dy,
\end{equation*}
where $G$ is the Green function of $D$. We know that $\nabla G\in L^q$ for all $q<2$ and we also have that the sequence $h_n-k_\infty$ converges to zero weakly in $L^p$ and is bounded in $L^p$. From the weak convergence we then obtain the pointwise convergence
\begin{equation}\label{convergpw}
 \nabla\psi_n(x)\to 0\text{ for all }x.
\end{equation}
We can then use the $L^p$-boundedness of $h_n-k_\infty$ together with the Young inequality
\begin{equation*}
 \|\nabla\psi_n\|_{L^r}\leq\|\nabla G\|_{L^q}\|h_n-k_\infty\|_{L^p},
\end{equation*}
for $q$ as above. We then have that $\|\nabla\psi_n\|_{L^r}$ are bounded once the following equivalent relations hold:
\begin{equation*}
 \frac{1}{r}>\frac{1}{p}+\frac{1}{2}-1\Leftrightarrow r<\frac{2p}{2-p},
\end{equation*}
In particular we have the boundedness in $L^r$ for some $r>p$. This together with the pointwise convergence \eqref{convergpw} and with the $L^p$-boundedness gives \eqref{normtozero}, as wanted.
\end{proof}

\section{Proof of Theorem \ref{mainthm}}\label{sec:proofmainthm}
Our strategy will be to apply Theorem \ref{hr} to the sequence $h_n$ arising from the $F_n$ of Theorem \ref{mainthm}. We start with two relatively elementary lemmas.
\begin{lemma}\label{egorov} Suppose that $d(h_n(t),h_\infty(t))\to0$ for almost all $t\in I$. Then for all $\alpha, \e>0$ there exists a subset $E_{\alpha, \e}\subset I$ such that $|E_{\alpha,\e}|<\e$ and that there exists $N_{\alpha,\e}$ such that $n>N_{\alpha,\e}$ and $t\in E_{\alpha,\e}$ imply
\begin{equation*}
 d(h_n(t),h_\infty(t))<\alpha
\end{equation*}
\end{lemma}
\begin{proof}
 Call $E_{m,n}:=\{x\in I:\;d(h_i(x),h_\infty(x))\leq1/m \text{ for } i\geq n\}$. Then for fixed $m_\alpha>\alpha^{-1}$, the sets $E_{m_\alpha,n}$ form an increasing sequence whose union is $I$. It follows that $|E_{m_\alpha,n}|\to|I|$, so we find $N_{\alpha,\e}$ such that $|I\setminus E_{m_\alpha,N_{\alpha,\e}}|\leq \e$. We then choose $E_{\alpha,\e}:=E_{m_\alpha,N_{\alpha,\e}}$. It is easy to verify that this set is as wanted.
\end{proof}
\begin{lemma}\label{intequal}
 Fix $x\in I$ and a $2$-form $h_\infty(x)$. For all $c>0$ there exists $\e>0$ such that 
\begin{equation*}
 \left.
\begin{array}{r}
 d(h(x), h_\infty(x))<\alpha\\
\int|h(x)|^p\leq A
\end{array}
\right\}\Rightarrow \int h(x)=\int h_\infty(x).
\end{equation*}
\end{lemma}
\begin{proof}
 Suppose by contradiction that there exists a $A>0$ such that for all $k\in\mathbb N$ there exists $h_k$ such that 
\begin{eqnarray*}
 d(h_k(x), h_\infty(x))&\leq& \frac{1}{k}\\
\int|h_k(x)|^p&\leq& A\\
\int h_k(x)&\neq&\int h_\infty(x)
\end{eqnarray*}
By the second property, we can extract a subsequence $h_{k'}(x)$ of the $h_k(x)$ converging weakly in $L^p$. In particular we would then have
$$
\mathbb Z\ni\int h_{k'}(x)\to \int h'_\infty(x)
$$
In particular, for some $N\in \mathbb N$ large enough, the subsequence $h_{k''}:=h_{k'+N}(x)$ satisfies 
$$
\int h_{k''}(x)=\int h'_\infty(x).
$$
We now prove that $h_\infty(x)=h'_\infty(x)$. It is enough to prove that $h_{k''}(x)\stackrel{d}{\to}h'_\infty(x)$. and this follows exactly as in the proof of Proposition \ref{ncpctsublev}. We thus contradicted the assumption $\int h_k(x)\neq\int h_\infty(x)$, as wanted.
\end{proof}
\begin{proof}[Proof of Theorem \ref{mainthm}]
By the $L^p$-boundedness of the $F_n$, it is clear that we may find a weakly converging subsequence $F_{n'}\stackrel{L^p}{\rightharpoonup}F_\infty$. We suppose by contradiction that there exists a point $x\in B_1(0)$ and two radii $0<s'<s<\op{dist}(x,\de B_1(0))$ such that 
\begin{equation}\label{contradmainthm}
\exists S\subset [s',s]\text{ s.t. } \mathcal H^1(S)>0 \text{ and }\forall t\in S,\quad\int_{\de B_t(x)}i^*_{\de B_t(x)}F\notin\mathbb Z.
\end{equation}
We then identify the forms given by
$$
\tilde F_n|_{\de B_r(x)}:=i^*_{\de B_r(x)}F\text{ for }r\in[s',s],
$$
with functions (defined almost everywhere) $h_n:[s',s]\to Y$ (with the notations of Section \ref{sec:aeconvthm}). We suppose without affecting the proof that $[s',s]=I$ (see also Remark \ref{rmkintervals}). By Theorem \ref{hr}, we can assume (up to extracting a subsequence) that there exists $h_\infty$ such that for almost all $t\in I$ there holds $d(h_n(t), h_\infty(t))\to 0$.\\

We call 
\begin{equation*}
 \left|\int h_n(x) - \int h_\infty(x)\right|:=f_n(t).
\end{equation*}
Since we have $f_n\geq 0$, if we prove that the $f_n$ converge in $L^1$-norm, then the almost everywhere pointwise convergence follows, implying the fact that $|S|=0$ and reaching the wanted contradiction. To prove Theorem \ref{mainthm} we therefore have to prove that
\begin{equation}\label{finalclaim}
 \lim_{n\to\infty}\int f_n(t)dt= 0.
\end{equation}
We start by calling 
$$
F_{n,A}:=\left\{t\in I:\;\int|h_n(t)|^p\geq A\right\}.
$$
It clearly follows that (with $C$ as in the statement of the theorem) 
$$
|F_{n,A}|\leq\frac{1}{A}\int \left(\int|h_n(t)|^p\right)dt=\frac{C}{A} 
$$
Now take $A$ such that the above quantity is smaller than $\e$, and use Lemma \ref{intequal} to obtain a constant $\alpha$ such that $d(h_n(t),h_\infty(t))<\alpha$ implies $f_n(t)=0$ for $t$ such that $\int|h_n(t)|^p< A$, i.e. for $t\notin F_{n,A}$. With such choice of $\alpha$ apply Lemma \ref{egorov} and obtain a set $E_{\alpha,\e}$ so that $|I\setminus E_{\alpha,\e}|<\e$ and an index $N_{\alpha,\e}$ such that for $n\geq N_{\alpha,\e}$ and for $t\in E_{\alpha,\e}$ there holds $d(h_n(t),0)<\alpha$, and therefore $f_n(t)=0$.\\

For $n>N_{\alpha,\e}$, the function $f_n(t)$ can therefore be nonzero only on $E_{\alpha,\e}\cup F_{n,A}$, and we have
\begin{eqnarray*}
\int f_n(t)dt&\leq& \int_{E_{\alpha,\e}\cup F_{n,A}}f_n(t)dt\\
&\leq& |E_{\alpha,\e}\cup F_{n,A}|^{1-1/p}\left[\int |f_n(t)|^pdt\right]^{1/p}\\
&\leq& (2\e)^{1-1/p}C,
\end{eqnarray*}
whence the claim \eqref{finalclaim} follows by the arbitrarity of $\e>0$, finishing the proof of our result.
\end{proof}

\section{The case $p=1$}\label{sec:remarks}
We prove here the result stated in the Main Threorem \ref{mainthm1} for $p=1$, thereby showing also that the thesis of Theorem \ref{mainthm} cannot hold when $p=1$. We consider the case when the domain is $[0,1]^3$ for simplicity. The case of general domains is totally analogous.
\begin{proposition}\label{counterex}
 Consider a signed Radon measure $X\in \mathcal M^3([0,1]^3)$, with total variation equal to $1$. Then there exists a family of vectorfields $X_k\in L_{\mathbb Z}^1$ such that
\begin{enumerate}
 \item\label{strangeproposition} There are two constants $0<c<C<\infty$ such that 
\begin{equation*}
 \left\{
\begin{array}{l}
 \forall k\quad c<\|X_k\|_{L^1([0,1]^3)}<C\\
\mathbb M(\op{div}X_k)\to\infty
\end{array}
\right.
\end{equation*}
\item\label{convascurrents} $\op{div}X_k=\de I_k$ for a sequence of integer rectifiable currents $I_k$ of bounded mass, and finally
\begin{equation*}
 X_k\rightharpoonup X 
\end{equation*}
\end{enumerate}
\end{proposition}
From the above, it immediately follows:
\begin{corollary}
 The class $\mathcal F_{\mathbb Z}^1$ is not closed by weak convergence.
\end{corollary}
The following holds for all $p<\frac{n}{n-1}$ in $n$ dimensions:
\begin{lemma}\label{dipoleconstr}
 Given a segment $[a,b]\subset\R{n}$ of length $\e>0$ and a number $\delta>0$, if $p<\frac{n}{n-1}$ then it is possible to find a vectorfield $X\in L^p(\R{n},\R{n})$ with
\begin{eqnarray*}
\op{div}X&=&\delta_a - \delta_b,\\
\op{spt}X&\subset&[a,b]+B_\e(0),\\
\|X\|_{L^p}&\leq&C\e^{n-(n-1)p}
\end{eqnarray*}
where $C$ is a geometric constant, and for two sets $A,B$, we denote $A+B:=\{a+b:\;a\in A,\;b\in B\}$.
\end{lemma}
\begin{proof}
 We may suppose that $a=(-\e,0,\ldots,0), b=(\e,0,\ldots,0)$ first. We then define the piecewise smooth 
\begin{equation*}
\left\{
\begin{array}{rcl}
 X(\pm\e(t-1), \e st)&=&\left(\frac{1}{\e t^{n-1}|B^{n-1}_1|}, \pm\frac{s}{\e t^{n-1}|B^{n-1}_1|}\right)\text{ for }(t,s)\in[0,1]\times B^{n-1}_1\\
X(x,y)&=&(0,0)\text{ if }|x|+|y|_{\R{n-1}}>\e.
\end{array}
\right.
\end{equation*}
Then clearly $\op{spt}X\subset [a,b]+B_\e$, and using the divergence theorem it is also easily shown that $\op{div}X=\delta_a-\delta_b$ in the sense of distributions. For the last estimate, we observe that 
$$
|X(x,y)|\leq \frac{C}{(\e - |x|)^{n-1}}\chi_{\{|x|+|y|\leq\e\}}(x,y),
$$
so we can estimate
\begin{eqnarray*}
\int_{\R{2}} |X|^p\;dx\;dy
&\leq&
C\int_0^\e\frac{(\e- x)^{n-1}}{(\e-x)^{(n-1)p}}\;dx\\
&=&
C\e^{n-(n-1)p}
\end{eqnarray*}
\end{proof}
\begin{proof}[Proof of Proposition \ref{counterex}]
We will do our construction first in the simpler model case $F=dy\wedge dz\llcorner[0,1]^3$. The modifications leading to the general case are treated separately.
\begin{itemize}
 \item \emph{The case of $X\equiv (1,0,0)$.}
We consider the collections of segments in $[0,1]^3$ given by
$$
\mathcal S_k:=\left\{\left[\left(-2^{-3k-1},0,0\right), \left(2^{-3k-1},0,0\right)\right]+(a,b,c)\;:\; (a,b,c)\in2^{-k}\mathbb Z^3\cap]0,1[^3\right\}.
$$
We then define an integral rectifiable $1$-current $I_k$ as the canonical integration from right to left along all the segments of $\mathcal S_k$. There clearly holds 
\begin{equation}\label{convmass}
\mathbb M(I_k)=2^{-3k}(2^k-1)^3\to 1,
\end{equation}
and it is a standard exercise in geometric measure theory (based on the approximation of $\mathcal H^3\llcorner[0,1]^3$ by sums of Dirac measures in the points $2^{-k}\mathbb Z^3\cap]0,1[^3$) to show that there holds:
\begin{equation}\label{convascurrents2}
I_k\rightharpoonup \mathcal H^3\llcorner[0,1]^3\otimes dx\simeq (1,0,0).
\end{equation}
We can then use Lemma \ref{dipoleconstr} for each one of the segments in $\mathcal S_k$ and with $\delta=\tfrac{1}{2}\e=2^{-3k}$ (which produces a set of $(2^k-1)^3$ vectorfields with disjoint supports, which can then be consistently extended to zero outside the set of the supports) each of whose $L^1$-norms is equal to $C\e^2\max\{\delta^{-1},\e^{-1}\}=2^{-3k}C$, which is proportional to the mass of the respective segment. Therefore (using \eqref{convascurrents2}), property \eqref{strangeproposition} follows.

The last point of the proposition follows by proving that also the vectorfields $X_k$ converge as $1$-currents to the diffuse current $X$. The strategy used is as the one usually adopted for the proof of the converence of the $I_k$: for a fixed smooth vectorfield $a$ and for $k\to\infty$ we may approximate 
\begin{eqnarray*}
 \langle X_k,a\rangle&:=&\sum_{\sigma\in\mathcal S_k}\int_{\op{spt}X^\sigma} X_k\cdot a\\
&=&\sum_{P\in \mathbb Z^3\cap]0,1[^3}\left[\left(\int_{\op{spt}X^\sigma} X_k(x)dx\right)\cdot a(P) \right. +\\
&& +\left.\int_{\op{spt}X^\sigma} X_k(x)\cdot Da(P)[x-P]dx\right]+ O_a(2^{-3k})\\
&=& \sum_{P\in 2^{-k}\mathbb Z^3\cap]0,1[^3} 2^{-3k}(1,0,0)\cdot a(P)+ O_a(2^{-3k})\\
&\to&\int_{[0,1]^3}a(x)\cdot (1,0,0)dx
\end{eqnarray*}
where the integral containing the differential $Da$ is zero by the symmetry properties of $X^\sigma$ and using the fact that
\begin{eqnarray*}
\frac{|O_a(\e)|}{\e}&\leq&\sup\left\{\left|\frac{a(x+\e u)-a(x)}{\e}-Da(x)[u]\right|:\;x\in B^3_1,\;u\in S^2\right\}\\
&\to& 0\text{ as }\e\to 0.
\end{eqnarray*}
\item \emph{The case of $X=(\rho,0,0)\in\mathcal M^3([0,1]^3)$, where $\rho$ is a probability density on $[0,1]^3$.} 
In this case we consider the $2^{3k}$ disjoint cubes $\mathcal C_k$ having the same centers as the segments in $\mathcal S_k$ and sidelength $2^{-k}$, and in the above construction we substitute to the segment $\sigma_k\in\mathcal S_k$ the segment $\sigma'_k$ having the same center, but length equal to $\rho(C_k)$, where $C_k\in\mathcal C_k$ is the cube with center equal to the one of $\sigma_k$ and $\sigma'_k$. The newly obtained currents $I'_k$ will still satisfy \eqref{convmass} and the analogous of \eqref{convascurrents2} given by:
$$
I'_k\rightharpoonup \rho\otimes dx\simeq (\rho,0,0).
$$
It is then easy to apply suitable modifications to the above proof showing that also in this case property \eqref{convascurrents} holds.
\item\emph{The general case.} We can write (by Radon-Nikodym decomposition):
$$
X=(\rho_1^+ - \rho_1^-,\rho_2^+ - \rho_2^-,\rho_3^+ - \rho_3^-),
$$
where $\rho_i$ are positive Radon measures of mass less than $1$. Doing separately the construction in the previous point for all the $\rho_i$ we obtain integer rectifiable currents $I_k$ of mass bounded by $6$, each of which is supported on finitely many segments. Applying Lemma \ref{dipoleconstr} to each of the above segments, we obtain vectorfields converging as before to the measure $X$, and since the supports of the vectorfields obtained in this way superpose not more than $6$ times, the estimate of the Lemma (used here for $p=1$) still holds, up to changing the constant.
\end{itemize}
\end{proof}

\bibliographystyle{amsalpha}
\bibliography{WeakClosureFinal}
\end{document}